\documentclass[10pt]{article}

\usepackage{graphics}
\usepackage{tikz}
\usepackage{amsfonts,mathrsfs, }
\usepackage{amsbsy}
\usepackage{amssymb}
\usepackage{amsmath,amsthm}
\usepackage[scr=rsfs,cal=boondox]{mathalfa}
\usepackage{verbatim}
\usepackage{amsfonts}
\usepackage{pst-all}
\usepackage{pstricks}
\usepackage[T1]{fontenc}
\usepackage{multicol}
\usepackage{color}
\usepackage{mathalfa}
\usepackage[autostyle]{csquotes}
\usepackage{pgf}
\usepackage[symbol]{footmisc}
\usepackage{mathtools}
\usepackage{hyperref}
\hypersetup{
    colorlinks=true,
    linkcolor=blue,
    filecolor=blue,      
    urlcolor=blue,
    citecolor=blue,
    pdftitle={Overleaf Example},
    pdfpagemode=FullScreen,
    }

\newtheorem{thm}{Theorem}
\newtheorem{problem}{Problem}
\newtheorem{prop}[thm]{Proposition}
\newtheorem{cor}[thm]{Corollary}
\newtheorem{defn}[thm]{Definition}
\newtheorem{lem}[thm]{Lemma}

\newtheorem{rem}[thm]{Remark}
\newtheorem{exm}[thm]{Example}

\newcommand{\E}{\mathcal{E}}

\newcommand{\G}{\mathcal{G}}
\newcommand{\Hy}{\mathcal{H}}

\newcommand{\W}{\mathcal{W}}

\newcommand{\bF}{\mathbb{F}}

\newcommand{\V}{\mathcal{V}}

\newcommand{\Z}{\mathbb{Z}}

\RequirePackage{pgfkeys}
\pgfkeys{/ytableau/options/.is family}

\pgfkeys{/ytableau/options, boxsize/.value required,  boxsize/.code = {%
\pgfkeysalso{nosmalltableaux}%
 \compare@YT{#1}{normal}%
 \ifeq@YT
 \xdef\macro@boxdim@YT{\expandonce@YT\boxdim@normal@YT}%
 \else
 \xdef\macro@boxdim@YT{#1}%
 \fi
 } }
 \pgfkeys{/ytableau/options, textmode/.value forbidden,
 textmode/.code = \set@textmode@YT,
mathmode/.value forbidden,
mathmode/.code = \set@mathmode@YT,
 }
 
\def\set@mathmode@YT{
 \gdef\skipin@YT{$}
 \gdef\skipout@YT{$}
 \def\smallfont@YT{\scriptstyle} }
 \set@mathmode@YT

\makeatletter
\let\@fnsymbol\@arabic
\makeatother
\date{}

\title{$L(2,1)$-labeling of some zero-divisor graphs associated with commutative rings}
\author {Rameez Raja$^{1}$ Annayat Ali$^{2}$ \footnote {$^{, 2}$ Department of Mathematics, National Institute of Technology Srinagar, Jammu and Kashmir, India.  Email: annayat\_05phd20@nitsri.net, rameeznaqash@nitsri.ac.in, Corresponding Author $^1$}}

\begin{document}
\maketitle
\begin{abstract} 
Let $\G = (\V, \E)$ be a simple graph, an $L(2,1)$-labeling of $\G$ is an assignment of labels from non-negative integers to vertices of $\G$ such that adjacent vertices get labels which differ by at least by two, and vertices which are at distance two from each other get different labels. The $\lambda$-number of $\G$, denoted by $\lambda(\G)$, is the smallest positive integer $\ell$ such that $\G$ has an $L(2,1)$-labeling with all labels as  members of the set $\{ 0, 1, \dots, \ell \}$. The zero-divisor graph of a finite commutative ring $R$ with unity, denoted by $\Gamma(R)$, is the simple graph whose vertices are all zero divisors of $R$ in which two vertices $u$ and $v$ are adjacent  if and only if $uv = 0$ in $R$. In this paper, we investigate $L(2,1)$-labeling of some  zero-divisor graphs. We study the \textit{partite truncation}, a graph operation that allows us to obtain a reduced graph of relatively small order from a graph of significantly larger order. We establish the relation between  $\lambda$-numbers of the graph  and its partite truncated one. We make use of the operation \textit{partite truncation} to contract the zero-divisor graph of a reduced ring to the zero-divisor graph of a Boolean ring.
\end{abstract}

\textit{2010 AMS Classification Code}: 05E40, 13A99.
   
\section{Introduction}
Given a simple undirected finite graph $\G$, and two positive integers $j, k$, an $L(j,k)$-labeling of $\G$ is a function $f: V(\G)\longrightarrow \Z_{\geq0}$ such that $|f(v_1) - f(v_2)| \geq j$ whenever $v_1$ is adjacent to $v_2$ and $|f(v_1)-f(v_2)| \geq k$ whenever $v_1$ and $v_2$  are at distance two apart. The difference between the maximum and minimum values of $f$ is called \textit{span} of $f$, denoted by $span(f)$. If $f$ is an $L(j,k)$-labeling of $\G$ with minimum value say $\delta$, then the function $g$ defined on vertices of $\G$ by $g(v) =f(v) - \delta$ is also an  $L(j,k)$-labeling with minimum value $0$ and the maximum value as $span(f) = span(g)$. As a result, we assume that every  $L(j,k)$-labeling has  $0$ as its minimum value and its span as the maximum value. The minimum span over all $L(j,k)$-labelings  of $\G$ is called the $L(j,k)$-labeling number, denoted by, $\lambda_{j,k}(\G)$.  An $L(j,k)$-labeling  $f$  of $\G$ is  said to be \textit{minimal labeling} if $span(f)= \lambda_{j,k}(\G)$. The $L(j,k)$-labeling problem has been studied extensively for the case $j=2$ and $k=1$. The $L(2,1)$-labeling number of a graph $\G$ is called its \textit{$\lambda$-number}, denoted by $\lambda(\G)$.

The difficulty of assigning channel frequencies to transmitters without interference motivates the study of $L(j,k)$-labeling problem of a graph. Roberts \cite{1} proposed the challenge of efficiently assigning radio channels to transmitters at several sites using non-negative integers to denote channels so that adjacent locations receive distinct channels, and extremely close locations receive channels that are at least at differ by two. As a result, there would be no interference between these channels. In 1992, Griggs and Yeh \cite{2} formalised the concept of $L(j,k)$-labeling and demonstrated that the $L(2,1)$-labeling problem is NP-complete for general graphs. The $L(j,k)$-labeling problem and $L(2,1)$-labeling problem have been studied in    \cite{3,4,5,6}.

On the other hand, the fundamental objective of associating a graph to an algebraic structure is to investigate the relationship between algebraic properties and combinatorial properties of algebraic structures. The graphs such as Cayley graph \cite{7}, power graph \cite{8}, zero-divisor graph \cite{9}, group-annhilator  \cite{10} and  torsion graph \cite{11}, are some graphs emerging from algebraic structures such as groups, rings and modules. Beck \cite{12} introduced the concept of an undirected zero-divisor graph $\Gamma^{\prime}(R)$ which is called \textit{Beck's zero-divisor graph} of a commutative ring $R$. In his investigation, all  elements of a ring $R$ are vertices of the graph $\Gamma^{\prime}(R)$ with two distinct vertices $u$ and $v$  adjacent in $\Gamma^{\prime}(R)$ if and only if $uv = 0$. Anderson and Livingston  \cite{9}  also studied the combinatorial properties of a commutative ring $R$. They associated a graph $\Gamma(R)$, called the \textit{zero-divisor graph}, to $R$ with vertices as elements of $Z^*(R) = Z(R)\setminus \{0\}$, that is, the non-trivial zero-divisors of $R$ with two vertices $u, v\in Z^*(R)$ being adjacent in $\Gamma(R)$ if and only if $uv = 0$. For more on zero-divisor graphs, please see \cite{ALS,B,FL,M,PBT,RR,Rd}. For basic definitions from graph theory we refer \cite{P}.

Let $R$ be a finite commutative ring with unity. A ring  $R$ is said to be a \textit{local ring} if it has a unique maximal ideal. Let  $R =R_1\times \dots \times R_r \times F_1 \times \dots \times F_s$,  be the Artinian decomposition of $R$,  where each $R_i$ is a commutative local ring with unity, and each $F_j$ is a field. A ring $R$ is said to be a \textit{mixed ring} if either $r\geq1$ and $s\geq1$ or $r\geq 2$ and $s = 0$ in its Artinian decomposition. A ring $R$ is said to be a \textit{reduced ring} if $r =0$ and $s \geq 1$ in its Artinian decomposition.

This research article is organised as follows. In Section 2, we present some preliminary results and definitions related to $L(2,1)$-labeling of graphs. In Section 3, we study \textit{partite truncation} operation in graphs, and obtain exact value of the $\lambda$-number of zero-divisor graphs realized by some classes of reduced rings. 

\section{Preliminaries}
In this section, we discuss some results related to \textit{$L(2,1)$-labeling} of some classes of graphs.

For a positive real number $d$, an \textit{$L_d(2,1)$-labeling} of a simple graph $\G$ is a positive real valued function $f$ defined on $V(\G)$ such that for $u_1, u_2\in V(\G), |f(u_1) - f(u_2)| \geq 2d$ whenever $u_1$ is adjacent to $u_2$ and $ |f(u_1) - f(u_2)| \geq d$ whenever distance between $u_1$ and $u_2$ is two. The $L_d(2,1)$-labeling number of $\G$, denoted by $\lambda(\G,d)$, is the smallest number $m$ such that $\G$ has an $L_d(2,1)$-labeling $f$ with $\max\limits_{v \in V(\G)}\{f(v)\} = m$. $L_d(2,1)$-labeling introduced by Griggs and Yeh \cite{2} is a natural generalisation of $L(2,1)$-labeling. Note that $\lambda(\G,1)=\lambda(\G)$. The authors confirmed that to determine $\lambda(\G,d)$, it suffices to study the case for $\lambda(\G)$. They provided bounds on $\lambda$-number for a variety of graphs, including trees, cycles, $3$-connected graphs and hypercubes. They proved following interesting results in which they obtain upper bounds of $\lambda(\G)$ in terms of graph invariants such as chromatic number $\chi(\G)$ and maximum degree $\Delta(\G)$.
\begin{thm}[\cite{2}, Theorem 4.1]
If $\G$ is a graph with $n$ vertices, then $\lambda(\G)\leq  n + \chi(\G) - 2$.
\end{thm}
\begin{thm}[\cite{2}, Theorem 6.2]
For any graph $\G$, $\lambda(\G)\leq \Delta^2(\G) + 2\Delta(\G)$.
\end{thm}
The authors in \cite{4} introduced the notion of holes, multiplicities and gaps of an $L(2,1)$-labeling. Let $f$ be an $L(2,1)$-labeling of $\G$, and  given a positive integer $h$, $0 < h < span(f)$, let $f_h=\{v\in V(\G): f(v) = h\}$. If $f_h$ is empty, then $h$ is called a \textit{hole} of $f$ and if $|f_h|\geq2$, then $h$ is called \textit{multiplicity} of $f$. $h$ is called a \textit{gap} of $f$ if $h$ is a hole with $|f_{h-1}| = |f_{h+1}| = 1$ and $\{v^{h-1},v^{h+1}\}\in E(\G)$, where $f(v^{h-1}) = h-1$ and  $f(v^{h+1}) = h+1$.  $H(f), M(f)$ and $G(f)$  represents collection of holes, multiplicities and gaps of $f$ respectively. Let $h(f)$ and $g(f)$ denote cardinalities of $H(f)$ and $G(f)$. The function $f$ is called the \textit{minimum $L(2,1)$-labeling}  of $\G$ if it is minimal and has minimum number of holes over the set $L_{(2,1)}(\G)$, where $L_{(2,1)}(\G)$ denotes the set of all minimal $L(2,1)$-labelings of $\G$. It should be noted that a minimal $L(2,1)$-labeling of a graph $\G$ is not unique whereas the minimum $L(2,1)$-labeling of a graph $\G$ is always unique. For example, if $P_6$ is a path on six vertices, then the labeling $(4,0, 3, 1, 5, 2)$ with span as 5 is an $L(2,1)$-labeling but not the minimum one, since it is not a minimal $L(2,1)$-labeling for $P_6$. 

 The graph invariant called as \textit{path covering number} is the least number of vertex-disjoint paths required to cover vertices of the graph. The relationship between $\lambda(\G)$ and path covering number $c(\G)$ of a graph $\G$ proved in \cite{4} is given as follows.
\begin{lem}[\cite{4}, Lemma 2.2]  Let $f$ be a minimum $L(2,1)$-labeling of $\G$. If $h$ is a hole of $f$, then $|f_{h-1}|=|f_{h+1}|>0$. Furthermore, if $|f_{h-1}|=|f_{h+1}|=1$, then $h$ is a gap.
\end{lem}
\begin{lem} [\cite{4}, Lemma 2.3] If $f$ is a  minimum $L(2,1)$-labeling of $\G$, then $G(f)$ is empty or $M(f)$ is empty. 
\end{lem}
\begin{thm}[\cite{4}, \mbox{Theorem 1.1}]
Let $\G$ be a simple graph of order n, and  let $\G^c$ be its complement. Then the following hold,\\
(i) $\lambda(\G)\leq n-1$ if and only if  $c(\G^c) = 1$.\\
(ii) Let $r\geq2$ be an integer. Then $\lambda(\G) = n+r-2$ if and only if $c(\G^c) = r$.

\end{thm}
A \textit{clique} of a graph $\mathcal{G}$ is a subset of  $V(\mathcal{G})$  such that any two distinct vertices of this subset are adjacent in $\mathcal{G}$. The maximum cardinality of a clique in $\mathcal{G}$ is known as its \textit{clique number}, denoted by $\omega(\mathcal{G})$. An \textit{independent set} of a graph $\G$ is a subset of $V(\G)$ such that no  two vertices in the subset are adjacent. The maximum cardinality of an independent subset in $\G$ is called its \textit{independence number}, denoted by $\alpha(\G)$. In \cite{15}, the authors investigated the relationships between $\lambda$-number and clique number, and also between $\lambda$-number and independence number of $\G$. They obtained some interesting results which are presented below.
\begin{thm}[\cite{15}, Proposition 4.1]
Let $C$ be a clique of a graph $\G$  such that $|C|=\omega(\G)$. Then $\lambda(\G) =2\omega(\G) -2$ if there exist partitions 
$$\{A_1, A_2, ~.~.~.~, A_s\} ~\mbox{and}~ \{C_1,C_2,~.~.~.~, C_s, C_{s+1}\}$$
of $V(\G)\setminus C$ and $C$ respectively satisfying  the  following conditions for each $i \in \{1,2, \dots , s\}$
\begin{enumerate}
\item $|A_i| \leq |C_i|-1$.
\item Every vertex in $A_i$ and every vertex in $C_i$ are non-adjacent in $\G$.
\end{enumerate}
\end{thm}
\begin{thm}[\cite{15}, Proposition 2.4]
 If $\G$ is a graph of order $n$, then $\lambda(\G) \leq 2n - \alpha(\G) -1$.
\end{thm}
A $p$-group is a group  in which order of every element is  a power of $p$, where $p$ is a prime number.  A finite group is a $p$-group if and only if  its order is a power of a prime $p$. The \textit{Power Graph} $\Gamma_G$  \cite{CGS, 8} of a finite group $G$ is the simple graph with vertex set $G$ in which two vertices are adjacent  if and only if one is the power of the other. The exact value of $\lambda$-number of power graphs of dihedral group $D_{2n}$, $p$-group, and generalised quaternion group $Q_{4n}$ was computed in \cite{15}. For some certain classes of groups of order $n$, they obtained an upper bound related to $\lambda$-number of a power graph.
\begin{thm}[\cite{15}, Theorem 4.1]
If $G$ is a group of order $n$, where $n$ is not a prime power, then $\lambda(\Gamma_G) \leq 2n-4$. The bound is achieved if and only if  $G$ is isomorphic to $\mathbb{Z}_2\times\mathbb{Z}_2$ or $\mathbb{Z}_{2q}$, where $q$ is a prime greater than $2$.
\end{thm}

Furthermore, the $L(j,k)$-labeling of Cayley graphs has been studied in \cite{13}, and Kelarev, Ras, and Zhou \cite{14} revealed a relationship between a semigroup structure and the minimum distance labeling spans of its Cayley graph.

\section{Partite truncation of a graph}
In this section, we study a graph operation called as \textit{partite truncation} of a graph. This operation allows us to obtain a graph of relatively smaller order  from a graph of significantly larger order. We prove an interesting result in which we determine the $\lambda$-number and a corresponding minimum labeling of Beck's zero-divisor graph $\Gamma^{\prime}(R)$ from the graph $\Gamma(R)$ of a ring $R$, provided that $diam(\Gamma(R))<3$.

\begin{defn}
Let $\G$ be an $n$-partite graph with  $\V_1, \V_2, \dots, \V_n$ as partite sets. We define \textbf{partite truncated graph}  $\hat{\G}$ of $\G$ with vertices  $\hat{v_1}, \hat{v_2}, \dots, \hat{v_n}$, where $\hat{v_i}$ corresponds to the whole partite set $\V_i$, $1\leq i \leq n$, and two distinct vertices $\hat{v_i}$ and $\hat{v_j}$ are adjacent in $\hat{\G}$ if and only if there is a vertex in $\V_i$ adjacent to some vertex in $\V_j$.
\end{defn}

Thus, in a partite truncation of an $n$-partite graph, each partite set is truncated to one vertex, and edges  (if any)  connecting two partite sets are truncated to a single edge.
It is clear from the definition that if $\G$ is connected, so is $\hat{\G}$, and if there is a path between two partite sets $\V_i$  and $\V_j$ of $\G$, then there must be a path between the two corresponding vertices $\hat{v_i}$  and $\hat{v_j}$  in $\hat{\G}$.
If $diam(\G) = 1$, then $ \G \cong \mathcal{K}_{n}  \cong  \hat{\G} $. Also, if the cardinality of each partite set of $\G$ is equal to one, then $\G \cong \hat{\G}$. We consider $n$-partite graphs $\G$ with $diam(\G) > 1$ in which there is at least one partite set $\V$ in $\G$ of  cardinality greater than one.

Let $\mathcal{H}$ denote the $n$-partite graph with the property that the bipartite subgraph induced by any two distinct partite sets $\V$ and $\mathcal{W}$ is either complete bipartite or empty. If $diam(\Hy) = 2$,  then there are some $u, w \in V(\Hy)$  such that $d(u, w)= diam(\Hy) $. If $u, w$ are in two  distinct partite sets $\V, \W$ of $\Hy$, then
$ d(\hat{v}, \hat{w}) = diam(\hat{\Hy})=2 $. However, if  there are no such vertices $u, w$  lying in two distinct partite sets $\V, \W$ of $\Hy$, then $diam(\hat{\Hy}) = 1$. Consequently, $diam(\hat{\Hy}) \leq 2 = diam(\Hy)$. If $diam(\Hy) \geq 3$, then it is easy to verify that $diam(\Hy) = diam(\hat{\Hy})$.

In the following result, we see that there is a shift  in the $\lambda$-number of a graph with a diameter less than three, when a particular number of isolated vertices (set of independent vertices ), and then  a dominant vertex (a vertex adjacent to all these isolated vertices and to all other  vertices of the graph) are added to it. This subsequently aids in determining the $\lambda$-number of Beck's zero-divisor graph $\Gamma^{\prime}(R)$ of a ring $R$.
\begin{prop}\label{z}
Let $\G$ be a graph with $diam(\G) <3$, and let the minimum labeling  of $\G$ have $s$ many holes. If $\G_{m+1}$ is the graph obtained from $\G$ when $m$ isolated vertices, and subsequently  a dominant vertex are added to it, then the following holds,
$$\lambda(\G_{m+1})=\begin{cases}
\lambda(\G)+m-s+2,~~~ m > s,\\
\lambda(\G)+2,~~~   ~~~~~~~~~~~ m \leq s.
\end{cases}$$

\end{prop}
\begin{proof}
 Let  $u_1, u_2, \dots , u_m$ denote the isolated vertices, and $u_0$ be the dominant vertex. Also, let $f$ be the minimum labeling of $\G$ with holes $h_1, h_2, \dots , h_s$. We define $L(2,1)$-labeling $g$ on $\G_{m+1}$ as,  $~g(u_0)= 0, ~~g(v) = f(v)+2, ~ \mbox{for all}~ v \in V(\G)$ and $ g(u_i)=  h_i+2,~\mbox{for all}~i,~  1 \leq i \leq s.$ For  $m > s, ~g(u_{s+j}) = \lambda(\G) + j +2,~\mbox{for all}~j,~1 \leq j \leq m-s$.

Therefore, 
$$span(g)= \begin{cases}
\lambda(\G)+m-s+2,~~~ m > s,\\
\lambda(\G)+2,~~~   ~~~~~~~~~~~ m \leq s.
\end{cases}$$
 Since $diam(\G)<3$, therefore $f$ must be one-to-one. This implies,  $g$ is well defined $L(2,1)$-labeling of $\G_{m+1}$. Thus, $\lambda(\G_{m+1}) \leq span(g)$. If we assume that there exist another $L(2,1)$-labeling $h$ of $\G_{m+1}$ with $span(h)<span(g)$, then  the restriction $h\restriction_{V(\G)}$  is a well defined $L(2,1)$-labeling of $\G$  with $span(h\restriction_{V(\G)})<span(f)$, a contradiction.
\end{proof}
\begin{cor} \label{cor}
Let $R$ be a finite ring of order $n$ such that  $diam(\Gamma(R)) < 3$, $\lambda(\Gamma(R)) = k$, and the minimum labeling of $\Gamma(R)$ have  $s$  many holes. Then the following holds, 
$$\lambda(\Gamma^{\prime}(R))=  \begin{cases}
k+n-m-s+1,  ~~~ n-m-1>s,\\
k+2,~~~~~~~~~~~~~~~~~~~n-m-1\leq s,
\end{cases}$$
 where $m$ is the order of $\Gamma(R)$.
\end{cor}
\begin{proof}
Any two distinct vertices $x,y \in R$ are adjacent in $\Gamma^{\prime}(R)$ if and only if $xy=0$. Therefore, from the structure of $\Gamma'(R)$, it is clear that the set of $n-m-1$ many units of $R$ represents the isolated vertices, and $0$ represents the dominant vertex added to  $\Gamma(R)$. Therefore, by Proposition~\ref{z}, $$\lambda(\Gamma^{\prime}(R))=  \begin{cases}
k+n-m-s+1,  ~~~ n-m-1>s,\\
k+2,~~~~~~~~~~~~~~~~~~~n-m-1\leq s.
\end{cases}$$
\end{proof}

Let $\hat{\Hy}$ be the partite truncation of the $n$-partite graph $\Hy$ with $diam(\Hy) \geq 3 $. Suppose $\lambda(\hat{\Hy})=k$ with the corresponding minimum labeling as $f$. Consider the set,
\begin{eqnarray*}
 \mathcal{F}&=&\{f(v):v \in V(\hat{\Hy})\}\\
 &=& \{ i_0, i_1, i_2, \dots , i_m \},
 \end{eqnarray*}
 where $f(v)=i_j$,  $0= i_0 < i_1 < \dots  < i_m = k$. Let $s_j$ denotes the  multiplicities  of $i_j, ~\mbox{for all} ~j,~0 \leq j \leq m$. For each $j$,  consider the subset $W_j = \{ v_{j, 1},  v_{j, 2}, \dots , v_{j, {s_j}}\}  = f^{-1}(i_j)
$ of vertices of $\hat{\Hy}$. Furthermore, for each $v_{j, l} \in W_j, 1\leq l \leq s_j$, let $\V_{j, l}$ be its corresponding partite set  in $\Hy$. Define the set $\mathcal{C}$, a collection of  partite sets of $\Hy$ by, 
\begin{eqnarray*}
\mathcal{C} &=& \{\V_0, \V_{1}, \V_{2}, \dots \V_m\},
\end{eqnarray*}
where, 
 \begin{eqnarray}
 \label{eqn}
\V_j &\in& \{\V_{j, l}:1\leq l \leq s_j\} \mbox{ and } |\V_j| = \max\limits_{1 \leq l \leq s_j}\{|\V_{j, l}|\}.
\end{eqnarray}

Now, we prove an important result related to $\lambda$-number of graphs $\Hy$ and $\hat{\Hy}$. We determine the $\lambda$-number of $\Hy$ if the $\lambda$-number of its partite truncation $\hat{\Hy}$ is given. This result also provides an algorithm for finding the minimal $L(2,1)$-labeling of $\Hy$  given the minimum $L(2,1)$-labeling of $\hat{\Hy}$.
\begin{thm}
\label{imp}
If $\Hy$ is an n-partite graph  and $\hat{\Hy}$ is its partite truncation,
then the following holds,
$$\lambda(\Hy) = \begin{cases}
|V(\Hy)| + \lambda(\hat{\Hy}) - n,  \text{ if }  diam(\Hy) = 2,\\
\sum\limits_{ \V_j \in \mathcal{C}} |V_j| + \lambda(\hat{\Hy}) - |\mathcal{C}|, \text{ if } diam(\Hy) \geq 3.
\end{cases}$$
\label{b} 
\end{thm}
\begin{proof} Let $f$ be the minimum $L(2, 1)$-labeling of the graph $\hat{\Hy}$. We consider the following two cases.\\
\textbf{Case - I : $diam(\Hy) = 2$.}\\
In this case, $f$ is injective. Therefore,
\begin{eqnarray*}
 \mathcal{F}&=&\{f(v):v \in V(\hat{\Hy})\}\\
 &=& \{0= i_0< i_1< i_2< \dots < i_m=\lambda(\hat{\Hy}) \}.
 \end{eqnarray*}

Here $s_j=1 ~\mbox{for all} ~j,~0 \leq j \leq n-1$. Let $v_j$ denote  the vertex in $\hat{\Hy}$ such that $f(v_j)=i_j$, and let $\V_j$ be its corresponding partite set in $\Hy$ of cardinality $m_j.$
Denote the vertices of $\Hy$ lying in the partite set $\V_j$ by $v_{j, i}$, where $1 \leq i \leq m_j$, that is,
\begin{eqnarray*}
\V_0 &=&\{ v_{0, 1}, v_{0, 2}, .~.~.~v_{0, m_0}\},\\
\V_1 &=&\{ v_{1, 1}, v_{1, 2}, .~.~.~v_{1, m_1}\},\\
&\vdots&\\
\V_{n-1} &=&\{ v_{{n-1}, 1}, v_{{n-1}, 2}, .~.~.~v_{{n-1}, m_{n-1}}\}.
\end{eqnarray*}
 Define $L(2,1)$-labeling  $g:V(\Hy) \longrightarrow \mathbb{Z}_{ \geq 0}$  on $\Hy$ by,
\begin{eqnarray*}
g(v_{0, i}) &=& i - 1,\\
g(v_{j, i})  &=& i + (i_j - i_{j-1}) + \max\limits_{v \in \V_{{j-1}}}\{ g(v)\} - 1,
\end{eqnarray*}
where,  $$\max\limits_{v \in \V_{_{j-1}}}\{ g(v)\} = m_0 + m_{1} + \dots + m_{{j-1}} + i_{j-1} -j.$$

Therefore, $span(g)   = |V(\Hy)| + \lambda(\hat{\Hy}) -n$. Since $f$ is well defined $L(2, 1)$-labeling of  $\hat{\Hy}$, vertices of $\Hy$ lying in the same partite set are at distance two apart, and the distance between two vertices of $\Hy$ lying in distinct partite sets is same as the distance between their corresponding vertices in $\hat{\Hy}$. This implies, $g$ is well defined $L(2,1)$-labeling of $\Hy$, and consequently  $\lambda(\Hy) \leq |V(\Hy)| + \lambda(\hat{\Hy}) -n$.

 Suppose $g$ is not a minimal $L(2,1)$-labeling, and let $h$ be a minimal $L(2,1)$-labeling of $\Hy$ with
$Span(h) < span(g).$ The operation partite truncation contracts each partite set to one vertex, and since the vertices lying in the same partite set are at distance two from each other, therefore, 
\begin{eqnarray*}
\lambda(\hat{\Hy}) &=& \lambda(\Hy) - \sum\limits_{ 0 \leq j \leq n-1} (|\V_j| - 1)\\
 &=& \lambda(\Hy)  - |V(\Hy)| + n\\
&<& |V(\Hy)| +  \lambda(\hat{\Hy}) -n - |V(\Hy)| + n\\
&=&  \lambda(\hat{\Hy})
\end{eqnarray*}
which clearly is a contradiction. Thus, $\lambda(\Hy) = |V(\Hy)| +  \lambda(\hat{\Hy})- n$.\\ 
\textbf{Case - II : $diam(\Hy) \geq 3$.}\\
We have, $diam(\hat{\Hy}) = diam(\Hy)   \geq 3$. This implies, $f$ may not be injective, which in turn implies that $s_j\geq1~\mbox{for all} ~j,~0 \leq j \leq m$. As above, 
\begin{eqnarray*}
 \mathcal{F}&=&\{f(v):v \in V(\hat{\Hy})\}\\
 &=& \{0= i_0< i_1< i_2< \dots < i_m=\lambda(\hat{\Hy}) \}.
 \end{eqnarray*}
 and
$\mathcal{C} = \{\V_0, \V_{1}, \V_{2}, \dots \V_m\}.$ Define $L(2,1)$-labeling $g$ on $V(\Hy)$ as follows, 
\begin{enumerate}
\item Label the partite sets in $\mathcal{C}$ in the same way as in case-I.
\item For each $j$, provide same labels for vertices in the remaining partite sets $\V_{j, l}$ (if any) as given to vertices in $\V_j \in \mathcal{C}$ (see \eqref{eqn} for the definition of  $\V_{j, l}$).
\end{enumerate}
On the similar lines as in case-I,  it can be shown that $g$ is a well defined  minimal  $L(2,1)$-labeling of $\Hy$, and we conclude that  $\lambda(\Hy) = \sum\limits_{ \V_j \in \mathcal{C}} |\V_j| + \lambda(\hat{\Hy}) - |\mathcal{C}|.$
\end{proof}
\begin{rem} The set $\mathcal{C}$ can be also defined by defining an equivalence relation $\sim$ on $V(\hat{\Hy})$ by $v_1 \sim v_2$ if and only if $f(v_1) = f(v_2)$, where $f$ is a minimal $L(2, 1)$-labeling of $\hat{\Hy}$. Since there is one-to-one correspondence between  $n$ vertices of $\hat{\Hy}$ and $n$ partite sets of $\Hy$, we have an equivalence relation on $n$-partite sets of $\Hy$ as well. We choose the representatives of each equivalence class in a way that it has maximum cardinality among other other partite sets in its class. The collection of all such partite sets among all equivalence classes defines $\mathcal{C}$.
\label{r}
 \end{rem}
Next, we provide some examples in which we compute the $\lambda-$number of a complete $n$-partite graph, and $\lambda$-number of  the zero-divisor graph of a ring $\mathbb{F}_p \times \mathbb{F}_q$, where  $\mathbb{F}_p$ and $\mathbb{F}_q$ are finite fields of order $p$ and $q$ respectively.
 \begin{exm}
 \label{imp1}
 Let $\mathbb{K}_{m_1,m_2, \dots, m_n}$ be a complete $n$-partite graph, and let $\V_i$ with $|\V_i| = m_i, 1\leq i\leq n$, be its partite sets. We assume that there exists at least one $ m_i$ such that $m_i>1$. This implies,  $diam(\mathbb{K}_{m_1, m_2, \dots , m_n}) = 2$. Note, that the partite truncation of  $\mathbb{K}_{m_1,m_2, \dots, m_n}$ is isomorphic to $\mathbb{K}_n$, and a bipartite subgraph induced by any two distinct partite sets $\mathcal{V_i}$ and $\mathcal{V_j}$ is a complete bipartite graph. It is easy to verify that $\lambda(\mathbb{K}_n) = 2n-2$. Thus, by Theorem~\ref{imp},  $\lambda(\mathbb{K}_{m_1,m_2, \dots, m_n})= \sum\limits_{i=1}^{n}m_i+n-2$.
 \end{exm}
\begin{exm}
Let $\mathbb{F}_p$ and $\mathbb{F}_q$ be two finite fields of order $p$ and $q$, respectively. The set of non-trivial zero divisors of $\mathbb{F}_p \times \mathbb{F}_q$ and hence the vertices of $\Gamma(\mathbb{F}_p \times \mathbb{F}_q)$ are given as $\V_1 = \{ (a,0) : a \in \mathbb{F}_p ~ \mbox{and} ~ a \neq 0\}$,
$\V_2 = \{ (0,b): b \in \mathbb{F}_q  ~ \mbox{and} ~ b \neq 0 \}$. This implies, $| \V_1| = p-1$ and $|\V_2| = q-1$.
There are no two vertices in sets $\V_1$ or $\V_2$ that are adjacent to each other, and each vertex in $\V_1$ is adjacent to every vertex in $\V_2$. Thus, $\Gamma( \mathbb{F}_p \times \mathbb{F}_q)  \cong \mathbb{K}_{p-1, q-1}$, and consequently,  $ \lambda(\G) = p+q -2$.
\end{exm}
 In the following result, we demonstrate that by applying the operation \enquote{partite truncation} to the zero-divisor graph of a reduced ring we obtain the zero-divisor graph associated with a Boolean ring.
 \begin{prop}
For $1 \leq i \leq n$, let $\bF_{q_i}$ be finite fields of order $q_i$ and let $\G \cong \Gamma(\prod\limits_{i = 1}^n \bF_{q_i})$. Then $\G$ is a $(2^n-2)$-partite graph with
 $\hat{\G} \cong  \Gamma(\prod\limits^n \Z_2).$
 \label{c}
 \end{prop} 
 \begin{proof}
 We prove the result in two parts parts. First, we show that $\G$ is $(2^n-2)$-partite graph, and then we show its partite truncation is isomorphic to $\Gamma(\prod\limits^n \Z_2)$.\\
  \textbf{Part-I}\\
Note that each non-trivial zero divisor of $\Gamma(\prod\limits_{i = 1}^n \bF_{q_i})$, and hence every vertex of $\G$ is an n-tuple which has at least  one zero  and one non-zero entry in it. For $1 \leq t < n$, we partition the vertices of $\G$ as follows,
 $$\mathcal{U}_t = \{ v_t : v_t \mbox{ has exactly }t \mbox{ non-zero entries in it}\}.$$
For each such $t$, the number of ways to choose $t$ entries out of $n$ entries is $\binom{n}{t}$. Furthermore, for each $t$, we partition the set $\mathcal{U}_t$ into $\binom{n}{t}$ sets as, $\mathcal{U}_t = \big\{\V_{l, t} : 1 \leq l \leq \binom{n}{t}\big\},$
where $\V_{l, t} $ consists of those vertices of $\mathcal{U}_t$ that have $t$ non-zero entries in the same position. For each $l$ and $t$, no two vertices in $\V_{l, t} $  are adjacent to each other. On the other hand, if a vertex of one such set say $\V_{l, r}$ is adjacent to some vertex of another set say $\V_{l', r'}$, where $1 \leq r \leq r' < n$, then each vertex of $\V_{l, r}$  is adjacent to every vertex of  $\V_{l', r'}$. The total number of all such sets in $V(\G)$ is given as,
  $\sum\limits_{1 \leq t < n} \binom{n}{t} = 2^n - 2.$ This implies, $\G$ is a $(2^n-2)$-partite graph.\\
  \textbf{Part-II}\\
As in Part-I, we partition vertex set of the graph $\Gamma(\prod\limits^n \Z_2)$ into $\big\{\V_{l, t}: 1 \leq l \leq \binom{n}{t}\big\}$. Since there is only one non-zero term in $\Z_2$, therefore, for each $l$ and $t$, $|\V_{l, t}|=1$. It follows that the order of  $\Gamma(\prod\limits^n \Z_2)$ is equal to the number of partite sets in  $ \Gamma(\prod\limits_{i=1}^n \bF_{q_i})$, which by Part-I is equal to $2^n-2$. The map $\hat{v_{l, t}}  \longrightarrow  \V_{l, t}$ defines an isomorphism between  $\hat{\G}$ and $\Gamma(\prod\limits^n \Z_2).$ 
\end{proof}
The following example illustrates the concept of partite truncation of a graph as well as Theorem~\ref{b} and Proposition~\ref{c}.
 \begin{exm}
 Let $\bF_{q_i}$ be finite fields of order $q_i$, where  $1 \leq i \leq 4$ and $q_1 \leq q_2 \leq q_3 \leq q_4$. As illustrated in Proposition~ \ref{c}, we partition vertex set of the graph $\Gamma(\prod\limits_{i=1}^4 \bF_{q_i})$ as,
  $ \mathcal{U}_1 = \big\{ \V_1, \V_2, \V_3, \V_4 \big\},
 \mathcal{U}_2= \big\{ \V_{12}, \V_{13}, \V_{14}, \V_{23}, \V_{24}, \V_{34}\big\}$ and $
 \mathcal{U}_3= \big\{ \V_{123}, \V_{124}, \V_{134}, \V_{234}\big\},$ 
  \begin{figure}\centering
   \begin{minipage}{0.48\textwidth}
    \hskip1cm  \includegraphics[scale=.19]{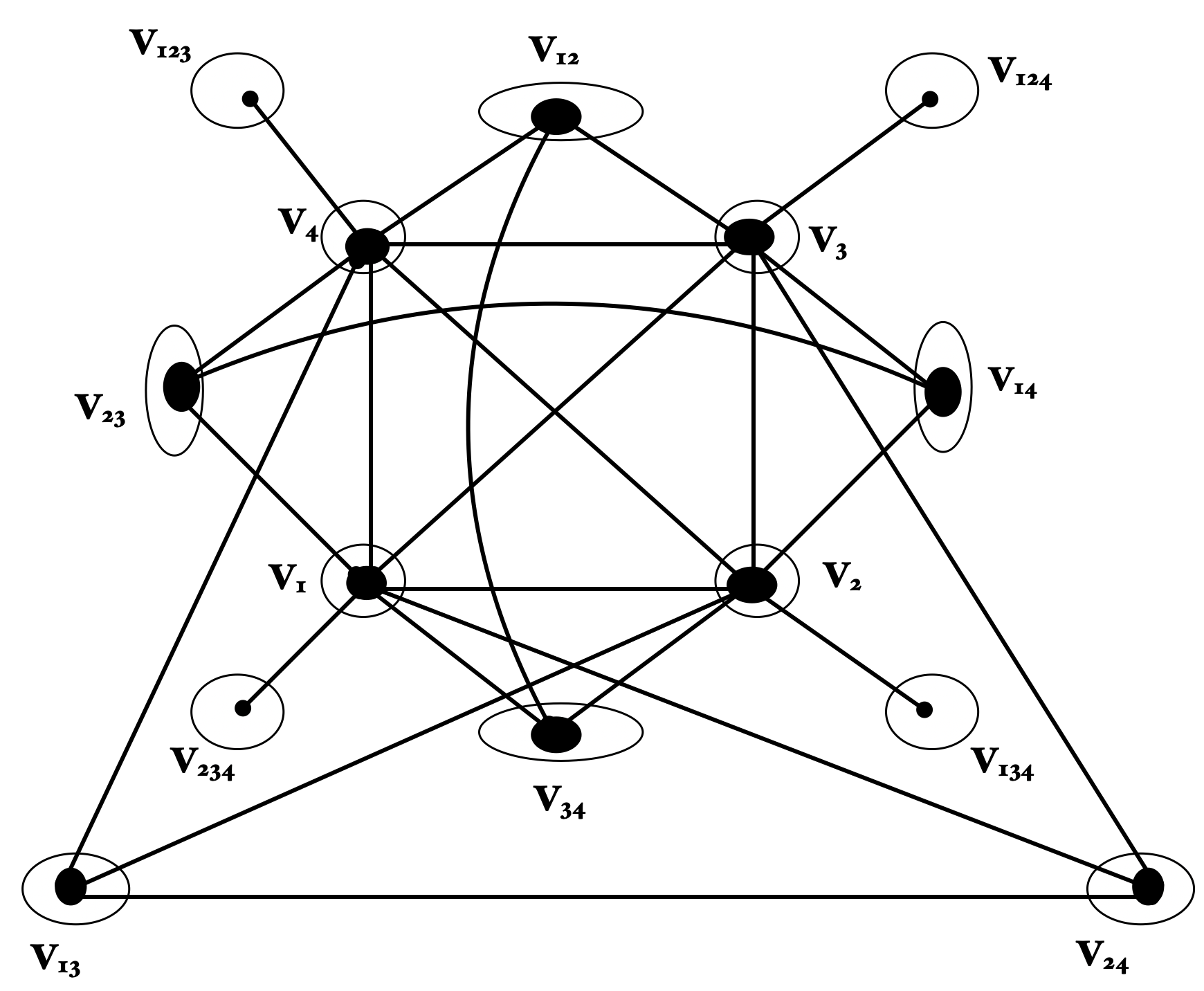}
     \caption{A Rough Sketch of $ \Gamma(\prod_{i=1}^4 \bF_{q_i})$}
 \label{Fig:1}
   \end{minipage}
   \begin {minipage}{0.48\textwidth}
    \hskip1cm  \includegraphics[scale=.26]{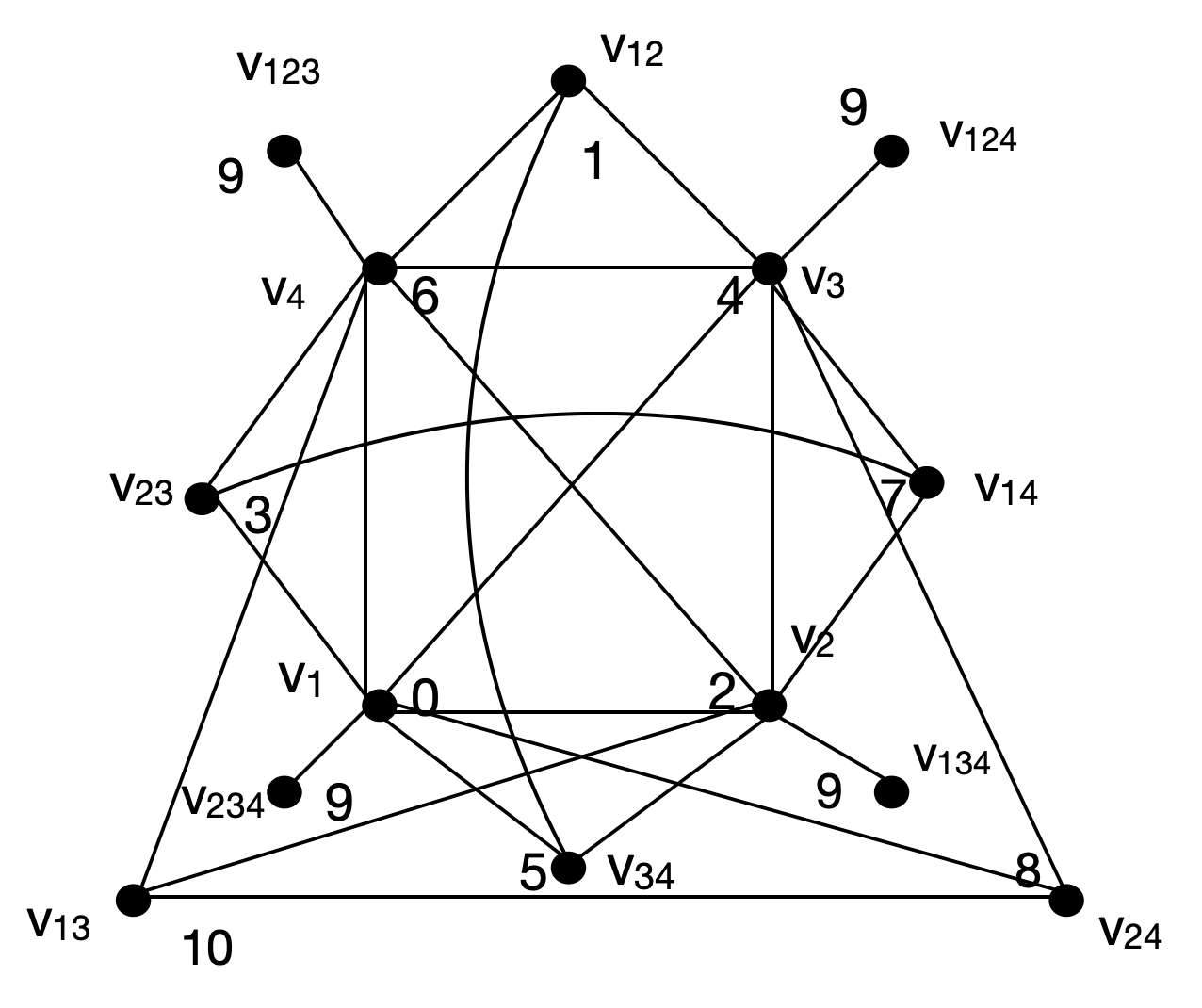}
      \caption{$\Gamma((\Z_2)^4)$ }
 \label{Fig:2}
   \end{minipage}
\end{figure}
and obtain $ \Gamma(\prod\limits_{i=1}^4 \bF_{q_i})$ as a $14$-partite graph. A rough drawing of $ \Gamma(\prod\limits_{i=1}^4 \bF_{q_i})$ is shown in Figure \ref{Fig:1}, where an edge between two partite sets means that each vertex of one set is adjacent to every vertex of the other. If we replace every non-zero entry of a vertex in the above partite sets with $1$, we obtain the partite truncation of the graph $ \Gamma(\prod\limits_{i=1}^4 \bF_{q_i})$, which is isomorphic to $\Gamma((\Z_2)^4)$. A minimal $L(2,1)$-labeling  of the zero-divisor graph  $\Gamma((\Z_2)^4)$ is shown in Figure 
\ref{Fig:2}. If we suppose that the labeling is minimal, then any other labeling with a shorter span would imply  that it has more than one multiplicity, which is impossible to achieve. Thus, 
$\lambda(\Gamma((\Z_2)^4))=10.$ Given the minimal labeling (as shown in Figure \ref{Fig:2} ) on the partite truncation of $ \Gamma(\prod\limits_{i=1}^4 \bF_{q_i})$,  we obtain the equivalence relation  on the partite sets of $ \Gamma(\prod\limits_{i=1}^4 \bF_{q_i})$ as discussed in Remark~\ref{r}. The representatives of the equivalence classes are given as, $\mathcal{C}= \big\{ \V_1, \V_2, \V_3, \V_4, \V_{12}, \V_{13}, \V_{14}, \V_{23}, \V_{24}, \V_{34}, \V_{234} \big\}.$ Therefore, by Theorem~\ref{b}, 
 \begin{eqnarray*}
 \lambda(\Gamma(\prod_{i=1}^4 \bF_{q_i})) & = & \sum\limits_{ \V \in \mathcal{C}} |\V| + \lambda( \Gamma((\Z_2)^4)) - |\mathcal{C}|\\
 &=& \sum\limits_{i=1}^4(q_i -1) + \sum\limits_{ 1 \leq i<j \leq 4}\big((q_i -1)(q_j -1) \big) + (q_2 - 1)(q_3-1)(q_4-1) -1.
 \end{eqnarray*}
  \end{exm}
It would be very interesting to determine the $\lambda$-number and its corresponding minimal $L(2, 1)$-labeling of the graph $\Gamma(\prod\limits_{i=1}^n \bF_{q_i})$. So, we conclude the paper with the following open problem.
 \begin{problem}
 Determine $\lambda$-number of the graph $\Gamma(\prod\limits^n(\Z_2))$ and its corresponding minimal $L(2,1)$-labeling.
 \end{problem} 
 
{\bf Conclusion}

In this research article, we studied partite truncation, a graph operation that aids to obtain a smaller order graph from relatively larger order $n$-partite graph. This contraction of the order allow us to determine the $\lambda$-number and  corresponding minimal labeling of the $n$-partite graph from the contracted graph. We exhibited that the zero-divisor graph associated with a Boolean ring is a partite truncated graph of the zero-divisor graph realized by a reduced ring.  We also  proved an interesting result that illustrates the deduction of $\lambda$-number of Beck's zero-divisor graph $\Gamma^{\prime}(R)$ from $\lambda$-number of the zero-divisor graph $\Gamma(R)$. We concluded this paper with an open problem for future research work.

{\bf Acknowledgement}
 
The first author's research is  supported by the University Grants Commission, Govt. of India under UGC-Ref. No. 191620023547, and the second author’s research work is funded by the Department of Atomic Energy, Govt. of India under S.No. 02011/15/2023NBHM(R.P)/R\&D II/5866.
 
\textbf{Declaration of competing interest}

There is no conflict of interest to declare.

\textbf{Data Availability}

Data sharing not applicable to this article as no datasets were generated or analysed during the current study.

\end{document}